\documentclass[10pt]{amsart}
\usepackage{latexsym, amsmath,amssymb}

\usepackage{hyperref}

\renewcommand{\epsilon}{\varepsilon}

\theoremstyle{plain}
\newtheorem{theorem}{Theorem}

\newtheorem{lemma}[theorem]{Lemma}
\newtheorem{corr}[theorem]{Corollary}
\newtheorem{prop}[theorem]{Proposition}
\newtheorem{deff}[theorem]{Definition}
\newtheorem{remark}[theorem]{Remark}
\newtheorem{conj}[theorem]{Conjecture}
\newcommand{\bth}{\begin{theorem}}
\newcommand{\ble}{\begin{lemma}}
\newcommand{\bcor}{\begin{corr}}

\newcommand{\bdeff}{\begin{deff}}

\newcommand{\bprop}{\begin{proposition}}
\newcommand{\ele}{\end{lemma}}
\newcommand{\ecor}{\end{corr}}
\newcommand{\edeff}{\end{deff}}

\newcommand{\eprop}{\end{proposition}}

\newcommand{\la}{\lambda}

\newcommand{\eps}{\varepsilon}

\newcommand{\supp}{\text{supp }}
\renewcommand{\Pi}{\varPi}

\renewcommand{\epsilon}{\varepsilon}

\newcommand{\ls}{\lesssim}
\newcommand{\gs}{\gtrsim}
\newcommand{\1}{{\rm 1\hspace*{-0.4ex}%
		\rule{0.1ex}{1.52ex}\hspace*{0.2ex}}}

\numberwithin{equation}{section}

\begin{document}
		\title[Logvinenko-Sereda sets and Carleson measures]{ $L^p$-Logvinenko-Sereda sets and $L^p$-Carleson measures on compact manifolds}
		\keywords{Carleson measure, Logvinenko-Sereda set, eigenfunction, heat kernel}
		\subjclass[2010]{35P99, 58C35, 58C40}
\author{Xing Wang, Xiangjin Xu and Cheng Zhang}
\address{School of Mathematics, Hunan University, Changsha, HN 410012, China}
\email{xingwang@hnu.edu.cn}

\address{Department of Mathematics and Statistics, Binghamton University-SUNY, Binghamton, NY, 13902, USA}
	\email{xxu@math.binghamton.edu}

	\address{Yau Mathematical Sciences Center,
	Tsinghua University,
	Beijing, BJ 100084, China}
\email{czhang98@tsinghua.edu.cn}
		\begin{abstract}
		Marzo and Ortega-Cerd\`a \cite{MO08} gave geometric characterizations for $L^p$-Logvinenko-Sereda sets on the standard sphere for all $1\le p<\infty$. Later, Ortega-Cerd\`a and Pridhnani \cite{OCP13} further investigated $L^2$-Logvinenko-Sereda sets and $L^2$-Carleson measures on compact manifolds without boundary. In this paper, we  characterize $L^p$-Logvinenko-Sereda sets and $L^p$-Carleson measures on compact manifolds with or without boundary for all $1<p<\infty$.   Furthermore, we investigate $L^p$-Logvinenko-Sereda sets and $L^p$-Carleson measures for eigenfunctions on compact manifolds without boundary, and we completely characterize them on the standard sphere $S^m$ for $p > \frac{2m}{m-1}$. For the range $p < \frac{2m}{m-1}$, we conjecture that $L^p$-Logvinenko-Sereda sets for eigenfunctions on the standard sphere $S^m$ are characterized by the tubular geometric control condition and we provide some evidence. These results provide new progress on an open problem raised  by Ortega-Cerd\`a and Pridhnani.  
		\end{abstract}
		\maketitle

\section{Introduction}

		Let $(M,g)$ be a smooth connected compact  Riemannian manifold with boundary $\partial M$. The dimension of $M$ is $m\ge2$. The boundary $\partial M$ is smooth and can be empty.  If $\partial M\neq \emptyset$ and  $\nu$ is the outward unit normal vector field along $\partial\Omega$, we assume the boundary condition is either  Dirichlet ($u|_{\partial\Omega}=0$) or Neumann ($\partial_\nu u|_{\partial\Omega}=0$).  Then  the Laplacian operator $-\Delta$ on $M$ is nonnegative, self-adjoint and has discrete spectrum $\{\la_j^2\}$ where $0\le \la_1\le \la_2\le ...$ are arranged in increasing order and we account for
		multiplicity. For $\la\ge1$, let $E_\la$ be the subspace of $L^2(M)$ generated by eigenfunctions of eigenvalues bounded by $\la^2$, namely
		\[E_\la=\Big\{\sum_{\la_j\le \la}\alpha_je_j:\alpha_j\in \mathbb{C},\ -\Delta e_j=\la_j^2e_j\Big\}\]
		where  $\{e_j\}_j$ is an orthonormal eigenbasis corresponding to  $\{\la_j\}_j$. 
		
		Let $1\le p<\infty$. The question is to characterize the measures $\mu=\{\mu_\la\}_\la$ such that for all $\la\ge1$ and $f\in E_\la$,
		\begin{equation}\label{goal}
			\int_M|f|^pd\mu_\la \approx \int_M|f|^pdV.
		\end{equation}
		Let $B(z,r)$ be the geodesic ball in $M$ centered at $z$ of radius $r$.  Heuristically, since the function $f$ is frequency-localized  in a ball of radius $\la$, by the uncertainty principle it behaves like a constant in every ball of radius $1/\la$ in the physical space. Thus, it is natural to introduce the following conditions, which will be useful to characterize the measures satisfying \eqref{goal}.

		We call $\mu$ is \textbf{relatively dense} if there exist $r>0$ and $\rho>0$ such that for all $\la\ge1$ and $z\in M$,
		\begin{equation}\label{mudense}
			\frac{\mu_\la(B(z,r/\la))}{\text{vol}(B(z,r/\la))}\ge \rho.
		\end{equation}
		 	We call $\mu$ is \textbf{relatively sparse} if there exist $r>0$ and $C>0$ such that for all $\la\ge1$ and $z\in M$,
		 \begin{equation}\label{musparse}
		 	\frac{\mu_\la(B(z,r/\la))}{\text{vol}(B(z,r/\la))}\le C.
		 \end{equation}

		 First, let us focus on the special case where $d\mu_\la=\1_{A_\la}dV$ where $\1_{A_\la}$ is the indicator function of the set $A_\la\subset M$.	Let $\mathcal{A}=\{A_\la\}_\la$ be a sequence of sets in $M$. We say that $\mathcal{A}$ is \textbf{$L^p$-Logvinenko-Sereda}  if there exists  $C>0$ such that for all $\la\ge1$ and $f\in E_\la$,
		 	\begin{equation}\label{LSdef}
		 	\int_{A_\la}|f|^pdV\ge  C	\int_M |f|^pdV.\end{equation}
		 		 We say that $\mathcal{A}$ is relatively dense if there exist $r>0$ and $\rho>0$ such that for all $\la\ge1$ and $z\in M$,
		 		\begin{equation}\label{RDdef}
		 			\frac{{\rm vol}(A_\la\cap B(z,r/\la))}{{\rm vol}(B(z,r/\la))}\ge\rho.
		 		\end{equation}
	By definition, $\mathcal{A}$ is relatively dense if and only if $\mu$ is relatively dense.

Our first result is  that  the relatively dense condition actually  characterizes $L^p$-Logvinenko-Sereda  sets on compact manifolds with or without boundary.

\begin{theorem}\label{LSthm}
	For $1<p<\infty$, the sequence of sets $\mathcal{A}=\{A_\la\}_\la$ is $L^p$-Logvinenko-Sereda  if it is relatively dense. For $1\le p<\infty$, the converse is true.
\end{theorem}

The classical Logvinenko-Sereda  theorem describes this type of comparable norms for functions in the Paley-Wiener space, which consists of functions in $L^p(\mathbb{R}^m)$ whose  Fourier transform is supported  in a prefixed bounded set in $\mathbb{R}^m$. See Logvinenko-Sereda \cite{LS74}, Havin-Jöricke \cite[p. 112-116]{HJ94}. Luecking \cite{L83} studied this notion in Bergman spaces. 
Following his idea, Marzo and Ortega-Cerd\`a \cite{MO08} characterized $L^p$-Logvinenko-Sereda sets on the sphere for all $1\le p<\infty$. Later, Ortega-Cerd\`a and Pridhnani \cite{OCP13} characterized $L^2$-Logvinenko-Sereda sets on compact manifolds without boundary. Our Theorem \ref{LSthm} completely characterizes $L^p$-Logvinenko-Sereda sets on compact manifolds with or without boundary for all $1<p<\infty$.

Next, let $1\le p<\infty$ and let $\mu=\{\mu_\la\}_\la$ be a sequence of measures on $M$. We say that $\mu$ is  \textbf{$L^p$-Carleson}  if there exists $C>0$ such that for all $\la\ge1$ and $f\in E_\la$,
\begin{equation}\label{cardef}
	\int_M |f|^pd\mu_\la\le C\int_M|f|^pdV.
\end{equation}

Similarly, we may characterize the $L^p$-Carleson measures by the relatively sparse condition. Nevertheless, it is sensitive to the boundary condition.
\begin{theorem}\label{carthm}
For $1<p<\infty$,  the sequence of measures $\mu=\{\mu_\la\}_\la$ is $L^p$-Carleson  if it is relatively sparse. For $1\le p<\infty$, the converse is true when the boundary is empty or the  boundary condition is Neumann, but the converse is false when the  boundary condition is Dirichlet.
\end{theorem}
 On compact  manifolds without boundary, Ortega-Cerd\`a and Pridhnani \cite{OCP13} proved that the relatively sparse condition is sufficient and necessary for $\mu$ to be $L^2$-Carleson. Later, Xu \cite{xu21} studied this problem on compact manifolds with boundary, and observed that the relatively sparse condition is not necessary for $L^2$-Carleson under Dirichlet boundary condition. We observe that due to   the conservation property of the heat kernel \eqref{conser},  the relatively sparse condition is necessary for Carleson measures in  the boundaryless case and the Neumann case.
 
Furthermore, we investigate the problem with  $E_\la$ replaced by the $\la$-eigenspace. This open problem was raised by Ortega-Cerd\`a and Pridhnani in \cite[Remark 5.12]{OCP13}. It is  much more complicated, so we  only consider  compact manifolds without boundary for now. We say that $\mathcal{A}=\{A_\la\}_\la$ is \textbf{$L^p$-Logvinenko-Sereda for eigenfunctions} on $M$  if there exists $C>0$ such that for all  $\la\ge1$ and eigenfunctions $e_\la$ satisfying $-\Delta e_\la=\la^2e_\la$ we have \begin{equation}\label{LSdef2}
	\int_{A_\la}|e_\la|^pdV\ge C\int_M |e_\la|^pdV.
\end{equation}
In other words, every eigenfunction $e_\la$ is $L^p$-concentrated on $A_\la$. For example, on the compact hyperbolic surface $M$, every eigenfunction  is $L^2$-concentrated on any fixed nonempty open set $\Omega\subset M$.  See e.g. Dyatlov-Jin \cite{DJ18}, Dyatlov-Jin-Nonnenmacher \cite{DJ21}. By Theorem \ref{LSthm}, the relatively dense condition  is sufficient for $\mathcal{A}$ to be  $L^p$-Logvinenko-Sereda  for eigenfunctions for all $1<p<\infty$, but it is not necessary for some $p$. See \cite[Remark 5.12]{OCP13} for a counterexample when $p=1$.   More precisely, we can show that it is not necessary whenever  $p$ is small. 

\begin{prop}\label{smallpls}
	For any $1\le p<\frac{2m}{m-1}$, the relatively dense condition  is not necessary for $\mathcal{A}$ to be  $L^p$-Logvinenko-Sereda  for eigenfunctions on $M$. For $p=\frac{2m}{m-1}$, the relatively dense condition  is not necessary for $\mathcal{A}$ to be  $L^p$-Logvinenko-Sereda  for eigenfunctions on $M$ with a generic metric.
\end{prop}
Here we use the fact that $\|e_\la\|_{L^\infty(M)}=o( \la^{\frac{m-1}2})\|e_\la\|_{L^2(M)}$ for a generic metric on any compact manifold. For the definition of generic metric and more related results, see Sogge-Zelditch \cite[Theorem 1.4]{sz01}, Sogge-Toth-Zelditch \cite{stz11}, Sogge-Zelditch \cite{sz16,sz162}.  

On the sphere $S^m=\{x\in\mathbb{R}^{m+1}:|x|=1\}$, we call $\mathcal{A}$ is \textbf{symmetric relatively dense} if there exist $r>0$ and $\rho>0$ such that for all $\la\ge1$ and $z\in S^m$,
\begin{equation}\label{RDdefsym}
	\frac{\text{vol}(A_\la\cap B(z,r/\la))+\text{vol}(A_\la\cap B(-z,r/\la))}{\text{ vol}(B(z,r/\la))}\ge\rho.
\end{equation}
It is weaker than the relatively dense condition.  However,  this condition is still not necessary for small $p$.
\begin{prop}\label{smallplssphere}
	Let $(S^m,g)$ be the sphere equipped with a Riemannian metric. For any $1\le p<\frac{2m}{m-1}$, the symmetric relatively dense condition  is not necessary for $\mathcal{A}$ to be  $L^p$-Logvinenko-Sereda  for eigenfunctions on $S^m$. For $p=\frac{2m}{m-1}$, the symmetric relatively dense condition  is not necessary for $\mathcal{A}$ to be  $L^p$-Logvinenko-Sereda  for eigenfunctions on $S^m$ with a generic metric.
\end{prop}
Nevertheless, we can characterize $L^p$-Logvinenko-Sereda sets for eigenfunctions on the standard sphere for large $p$ by the symmetric relatively dense condition.
\begin{theorem}\label{largepls}
	For any $\frac{2m}{m-1}<p<\infty$, the symmetric relatively dense condition  is  necessary and sufficient for $\mathcal{A}$ to be  $L^p$-Logvinenko-Sereda sets for eigenfunctions on the standard sphere $S^m$.
\end{theorem}
By Proposition \ref{smallplssphere}, the range of $p$ in Theorem \ref{largepls} is essentially sharp. For the range $1\le p<\frac{2m}{m-1}$, we propose a conjecture that $L^p$-Logvinenko-Sereda sets for eigenfunctions on the standard sphere can be characterized by the tubular geometric control condition. See Conjecture \ref{tubeconj} and  some evidence in Section \ref{gausssect}.

 Similarly, it is interesting to characterize $L^p$-Carleson measures for eigenfunctions. We say that $\mu=\{\mu_\la\}_\la$ is  \textbf{$L^p$-Carleson for eigenfunctions} if there exists $C>0$ such that for all  $\la\ge1$ and eigenfunctions $e_\la$ satisfying $-\Delta e_\la=\la^2e_\la$ we have
\begin{equation}\label{cardef2}
	\int_M |e_\la|^pd\mu_\la\le C\int_M|e_\la|^pdV.
\end{equation} 
By Theorem \ref{carthm}, the relatively sparse condition is sufficient for $\mu$ to be $L^p$-Carleson for eigenfunctions, but we observe that it is not necessary for small $p$.

\begin{prop}\label{smallp}
	For any $1\le p<\frac{2m}{m-1}$, the relatively sparse condition  is not necessary for $\mu$ to be $L^p$-Carleson for eigenfunctions on $M$. For $p=\frac{2m}{m-1}$, the relatively sparse condition  is not necessary for $\mu$ to  be $L^p$-Carleson for eigenfunctions on $M$ with a generic metric.
\end{prop}
On the other hand, we are able to characterize $L^p$-Carleson measures for eigenfunctions on the standard sphere for large $p$ by the relatively sparse condition.
\begin{theorem}\label{largep}
	 For  any $\frac{2m}{m-1}<p<\infty$, the relatively sparse condition  is  necessary and sufficient for $\mu$ to be $L^p$-Carleson for eigenfunctions on the standard sphere $S^m$.
\end{theorem}
By Proposition \ref{smallp}, the range of $p$ in Theorem \ref{largep} is essentially sharp. We essentially use the property that at every point of the sphere, there exists a sequence of eigenfunctions (e.g. zonal spherical harmonics)  with maximal eigenfunction growth (measured by  $L^\infty$ norm) at that point.  Sogge-Zelditch \cite[Theorem 1.1]{sz01} gave a necessary condition for this property using geodesic loops. It would be interesting to look for a sufficient condition.

\noindent\textbf{Outline for the proof of Theorem \ref{LSthm} and Theorem \ref{carthm}.} 

To prove the sufficiency part of  Theorem \ref{LSthm}, we first handle the boundaryless case and then deal with the Dirichlet case and the Neumann case separately. The boundaryless case can be proved by the harmonic extension method in Ortega-Cerd\`a and Pridhnani \cite{OCP13}. We need the $L^p$-mean value inequality for harmonic functions and the $L^p$-boundedness for the spectral multipliers. For the Dirichlet case, we find a good control  for the contribution for the part close to the boundary so we  eventually reduce the proof to  the boundaryless case. We need Bernstein inequalities that essentially rely on the Gaussian heat kernel upper bounds.  For the Neumann case, we need to establish the gradient estimate \eqref{gradest} for harmonic functions under Neumann boundary condition in order to proceed in the way as the boundaryless case. Our argument is inspired by the book of Li \cite{li12}, Chen \cite{chen90} and Wang \cite{wang97}.

To prove the necessity part,  heuristically one may  take $f=\1_{B(z,r/\la)}$ in \eqref{LSdef} to get \eqref{RDdef}, but this function is not in $E_\la$. Nevertheless, this heuristic can be made more precise. We shall use the complex time heat kernel estimates to construct appropriate functions in $E_\la$ that are essentially supported in the ball of radius $1/\la$. We also exploit Seeley's spectral function estimates and Bernstein inequalities.

 The proof of Theorem \ref{carthm} is based on the similar strategy. The sufficiency part is simpler and we only use the $L^p$-mean value inequality for the harmonic extensions and the $L^p$-boundedness of spectral multipliers. To prove the necessity part, we shall exploit the two-sided estimate for the heat kernel on the diagonal, which holds when the boundary is empty or the  boundary condition is Neumann. It is due to the conservation property of the heat kernel. We also need the complex time heat kernel estimates  to construct functions in $E_\la$ that  approximate the indicator function $f=\1_{B(\xi,1/\la)}$. Furthermore, we construct a counterexample when the boundary condition is Dirichlet.

\noindent\textbf{Remarks.}
 1. It is natural to consider the endpoint case $p=1$ in Theorem \ref{LSthm} and Theorem \ref{carthm}. Recall that Marzo and Ortega-Cerd\`a \cite{MO08} characterized $L^p$-Logvinenko-Sereda sets on the sphere for $1\le p<\infty$, by using classical estimates on the Jacobi polynomials. However, the endpoint case $p=1$ seems more subtle on general Riemannian manifolds, and our approach essentially relies on the $L^p$-mean value inequality \eqref{mv} for harmonic functions and  the $L^p$-boundedness of spectral multipliers in Lemma \ref{Lplemma}.

2. It is an interesting open problem to characterize $L^p$-Logvinenko-Sereda sets and $L^p$-Carleson measures for eigenfunctions on $M$ for $1\le p\le \frac{2m}{m-1}$, even if $M$ is the standard sphere. We can choose Gaussian beams as test functions to give some necessary conditions. See Section \ref{gausssect} for  a conjecture and some evidence on the standard sphere.

\noindent\textbf{Organization of the paper.} In Section 2, we introduce the estimates of the heat kernels and the spectral functions. In Section 3, we prove the sufficiency part of Theorem \ref{LSthm}. In Section 4, we prove the necessity  part of Theorem \ref{LSthm}. In Section 5, we prove Theorem \ref{carthm}. In Section 6, we prove Proposition \ref{smallpls}, Proposition \ref{smallplssphere} and Theorem \ref{largepls}.  In Section 7, we prove Proposition \ref{smallp} and Theorem \ref{largep}. In Section 8, we further discuss $L^p$-Logvinenko-Sereda sets and $L^p$-Carleson measures for eigenfunctions on the standard sphere for small $p$.

\noindent\textbf{Notations.} Throughout this paper, $X\ls Y$ means $X\le CY$  for some positive constants $C$ independent of $\la$.   If $X\ls Y$ and $Y\ls X$, we denote $X\approx Y$. We write $X\ll Y$ if $CX\le Y$  for some large positive constants $C$ independent of $\la$. Let $d(\cdot, \cdot)$ be the Riemannian distance function on $M$.  Let $dV$ be the volume element of $M$. Let $\text{vol}(E)$ be the volume of the set $E\subset M$. Let $B(z,r)$ be the geodesic ball in $M$ centered at $z$ of radius $r$, and $\mathbb{B}(z,r)$ be the Euclidean ball in $\mathbb{R}^m$ centered at $z$ of radius $r$.

\noindent\textbf{Acknowledgments.}
The authors would like to thank Nicolas Burq,  Long Jin and Shanlin Huang for  helpful discussions and comments during the research. X.X. would like to thank YMSC at Tsinghua University for the kind hospitality in June 2024.  X.W. is partially supported by the Fundamental Research Funds for the Central Universities Grant No. 531118010864 from Hunan University. C.Z. is partially supported by National Key R\&D Program of China No. 2024YFA1015300 and NSFC Grant No. 12371097. 


\section{Estimates of the heat kernels and the spectral functions} 
In this section, we review some useful estimates for the heat kernels and the spectral functions. We will use them throughout the paper. As before, let $M$ be a smooth compact manifold of dimension $m\ge2$. The boundary $\partial M$ is smooth and can be empty.  If $\partial M\neq \emptyset$, we assume the boundary condition is either  Dirichlet or Neumann.

First, we need  short time Gaussian upper bounds for the heat kernel and its gradient. 
\begin{lemma}\label{heatlemma}Let $e^{t\Delta}(x,y)=p(t,x,y)$ be the heat kernel. Then there exist constants $C>0$ and $c>0$ such that for all $t\in(0,1]$ and $x,y\in M$ we have	\begin{equation}\label{heat}
		|p(t,x,y)|\le Ct^{-\frac m2}e^{-cd(x,y)^2/t},
	\end{equation}
	\begin{equation}\label{heatgrad}
		|\nabla_xp(t,x,y)|\le Ct^{-\frac m2-\frac12}e^{-cd(x,y)^2/t}.
	\end{equation}
\end{lemma}
The pointwise estimates \eqref{heat} and \eqref{heatgrad} follow from parametrix constructions for $p(t,x,y)$ on $(0,1]\times M\times M$. See e.g. Greiner \cite{gre71}, Li-Yau \cite{liyau}. These estimates are useful to obtain the boundedness of spectral multipliers of generalized Laplacians on compact manifolds with boundary, see e.g. Taylor \cite{taylornote}, Mukherjee  \cite{mu18}. A direct consequence of \eqref{heat} is the upper bound for the spectral function.
\begin{corr}\label{corupp}
	 There exists $C>0$ such that for all $\la\ge1$ and $x,y\in M$
	\begin{equation}
		\Big|\sum_{\la_j\le \la}e_j(x)\overline{e_j(y)}\Big|\le C\la^m.
	\end{equation}
\end{corr}

When the boundary is empty or the boundary condition is Neumann, we have the conservation property \begin{equation}\label{conser}
	\int_M p(t,x,y)dV(y)= 1
\end{equation} for all $t>0$ and $x\in M$. Together with the semigroup property and the short-time Gaussian upper bound, the conservation property implies a short time lower bound for the heat kernel on the diagonal. See e.g.  Hebisch and  Saloff-Coste \cite[p. 1455]{HSC01}.

\begin{corr}\label{corlow}Suppose that the boundary is empty or the boundary condition is Neumann.  Then there exist $t_0>0$ and $C>0$ such that for all $t\in (0,t_0]$ and $x\in M$ we have
	\begin{equation}\label{twoheat}
		C^{-1}t^{-m/2}\le p(t,x,x)\le Ct^{-m/2}.
	\end{equation}
	As a consequence, there exists $C>0$ such that for all $\la\ge1$ and $x\in M$
\begin{equation}\label{twosideest}
	C^{-1}\la^m\le \sum_{\la_j\le \la}|e_j(x)|^2\le C\la^m.
\end{equation}
\end{corr}
\begin{proof}We just need to prove the lower bounds, as the upper bounds follow directly from Lemma \ref{heatlemma} and Corollary \ref{corupp}. By the semigroup property  and Cauchy-Schwarz, we get for any fixed $a>0$ 
	\begin{align*}
		p(t,x,x)&=\int_M p(t/2,x,y)^2dV(y)\ge \int_{B(x,a\sqrt t)}p(t/2,x,y)^2dV(y)\\
		&\ge\frac1{\text{vol}(B(x,a\sqrt t))}\Big(\int_{B(x,a\sqrt t)}p(t/2,x,y)dV(y)\Big)^{1/2}.
	\end{align*}
	By the conservation property \eqref{conser} and the Gaussian heat kernel bound \eqref{heat} we get
	\begin{align*}
	\int_{B(x,a\sqrt t)}p(t/2,x,y)dV(y)&=\int_{M}p(t/2,x,y)dV(y)-\sum_{j=0}^\infty\int_{B(x,2^{j+1}a\sqrt{t})\setminus B(x,2^{j}a\sqrt{t})}p(t/2,x,y)dV(y)\\
	&\ge 1-\sum_{j=0}^\infty C(t/2)^{-m/2}e^{-c2^{2j}a^2}\text{vol}(B(x,2^{j+1}a\sqrt{t}))\\
	&\ge 1-C_m\sum_{j=0}^\infty (2^j a)^m e^{-c2^{2j}a^2}\ge \frac12,
	\end{align*}
whenever $a>0$ is large enough. 

We fix such $a$ and then for sufficiently small $t$ we have the heat kernel lower bound
\[	p(t,x,x)\ge \frac1{4\text{vol}(B(x,a\sqrt t))}\gs t^{-m/2}.\]
Next, we use the heat kernel lower bound to obtain the lower bound for the spectral function
\begin{align*}
	\sum_{\la_j\le\la}|e_j(x)|^2&\ge \sum_{\la_j\le \la}e^{-\la_j^2/(c_0\la)^2}|e_j(x)|^2\\
	&\ge e^{\Delta/(c_0\la)^2}(x,x)-\sum_{\la_j>\la}e^{-\la_j^2/(c_0\la)^2}|e_j(x)|^2\\
	&\ge C^{-1}(c_0\la)^m-C\sum_{\ell\ge0}e^{-2^{2\ell}/c_0^2}(2^{\ell}\la)^m\\
	&\ge C^{-1}(c_0\la)^m-C_1\sum_{\ell\ge0}(2^\ell/c_0)^{-m-1}(2^{\ell}\la)^m\\
	&=C^{-1}(c_0\la)^m-2C_1c_0^{m+1}\la^m\\
	&\gs \la^m,
\end{align*}
whenever  $c_0>0$ is small enough.
\end{proof}
The lower bounds in Corollary \ref{corlow} are not valid when boundary condition is Dirichlet. See e.g.  Zhang \cite{zhangq02} for a different two-sided estimate of the Dirichlet heat kernel. Nevertheless,  we still have \eqref{twosideest} for the Dirichlet spectral function whenever the point $x$ is not too close to the boundary, see Corollary \ref{awaylow}.

  Seeley \cite{seeley1, seeley2} established the  spectral function estimate on compact manifolds with boundary for either Dirichlet or Neumann boundary condition.

\begin{lemma}\label{slem}
	There exist $C>0$ and $\alpha>0$ such that for all $\la\ge1$ and $x\in M$ we have
	\begin{equation}
		\Big|\sum_{\la_j\le \la}|e_j(x)|^2-\omega_m(2\pi)^{-m}\la^m\Big|\le C(\delta^{\alpha-1}\la^{m-1}+\delta^{-3/2}\la^{m-\frac32})
	\end{equation}
	where $\omega_m$ is the volume of the unit ball in $\mathbb{R}^m$ and  $\delta=d(x,\partial M)$ is the geodesic distance from $x$ to the boundary $\partial M$.
\end{lemma}
As a consequence, we get the following two-sided estimate for the spectral function.
\begin{corr}\label{awaylow}
	There exists $C>0$ and $C_1>0$ such that for all $\la\ge1$ and $x\in M$ with $d(x,\partial M)\ge C/\la$,
	\begin{equation}\label{twos}
		C_1^{-1}\la^m\le \sum_{\la_j\le \la}|e_j(x)|^2\le C_1\la^m.
	\end{equation}
\end{corr}
The two-sided estimate \eqref{twos} is well-known on compact manifolds without boundary and it follows from the pointwise Weyl law, see Avakumovi\'c \cite{ava}, Levitan \cite{lev1,lev2}, H\"ormander \cite{hor}, Sogge \cite{fio,hangzhou}.

Recall that the symbol class $S^0$ consists of smooth functions $\varphi:\mathbb{R}\to \mathbb{R}$ satisfying \begin{equation}\label{pdo}
	|\varphi^{(k)}(t)|\le C_k(1+|t|)^{-k},\ \ k=0,1,2,....
\end{equation}
\begin{lemma}\label{Lplemma}Let $1<p<\infty$ and $\varphi\in S^0$. Then $\varphi(\sqrt{-\Delta}):L^p(M)\to L^p(M)$ is bounded.
\end{lemma}
On compact manifolds without boundary, the boundedness of spectral multipliers in Lemma \ref{Lplemma} are well-known, see e.g. Seeger-Sogge \cite{ss89}.  For such results on compact manifolds with boundary, see Xu \cite{xu09,xu11}, Taylor \cite{taylornote}, Mukherjee  \cite{mu18}.  For results in the non-compact setting, see Cheeger-Gromov-Taylor \cite{cgt82},  Mauceri-Meda-Vallarino \cite{mmv09}, Taylor \cite{taylor09,taylor89,taylor009}. For such results in the general setting of a metric measure space, see e.g. Clerc-Stein \cite{cs74}, Duong-McIntosh \cite{dm99}, Duong-Ouhabaz-Sikora \cite{dos02}, Duong-Yan \cite{dy05,dy005}.

Furthermore,  by the semigroup property, we can obtain long time heat kernel estimates from Lemma \ref{heatlemma}.

\begin{corr}\label{longheat}For $b\ge0$, let $L=-\Delta+b$ and  $p(t,x,y)=e^{-tL}(x,y)$. Then  for sufficiently large $b$, we have \eqref{heat} and \eqref{heatgrad}  for all $t>0$ and $x,y\in M$.
\end{corr} 
\begin{proof}
	By Lemma \ref{heatlemma}, it suffices to consider $t>2$. Let $[t]$ be the integer part of $t$.  By \eqref{heat} we have $|e^{\Delta}(x,y)|\le C$. By the semigroup property, we have
	\[e^{-tL}=e^{t(\Delta-b)}=e^{-tb}(e^{\Delta})^{[t]}e^{(t-[t])\Delta}.\]
	This implies
	\[|e^{-tL}(x,y)|\le e^{-tb}C^{[t]+1}(\text{vol}(M))^{[t]},\]
	which still satisfies \eqref{heat} for $t>2$ if $b$ is large enough.
	
	Similarly, by \eqref{heatgrad} we have $|\nabla_x e^{\Delta}(x,y)|\le C$. Note that
\[\nabla e^{-tL}=e^{-tb}(\nabla e^{\Delta})(e^{\Delta})^{[t]-1}e^{(t-[t])\Delta}\]
 implies 
\[|\nabla_xe^{-tL}(x,y)|\le e^{-tb}C^{[t]+1}(\text{vol}(M))^{[t]},\]
	which still satisfies \eqref{heatgrad} for $t>2$ if $b$ is large enough.
\end{proof}

In the following, we always fix  $L$ as in Corollary \ref{longheat} such that  \eqref{heat} and \eqref{heatgrad} hold  for all $t>0$ and $x,y\in M$.   These long time heat kernel estimates imply the following Bernstein inequalities, see e.g. Imekraz-Ouhabaz \cite[Theorem 1.1, 1.2, 2.2]{IO22}.

\begin{lemma}\label{corbern}
 There exists $C>0$ such that for all $1\le p\le \infty$, $\la\ge1$ and $f\in E_\la$
	\begin{equation}\label{bern}
		\|\nabla f\|_{L^p(M)}\le C \la\|f\|_{L^p(M)}.
	\end{equation}
Moreover, for any $\psi\in C_0^\infty(\mathbb{R})$, there exists $C>0$ such that for all $1\le p\le \infty$, $\la\ge1$ and $f\in E_\la$
	\begin{equation}\label{bern1}
		\|\nabla \psi(L/\la^2)f\|_{L^p(M)}\le C \la\|f\|_{L^p(M)}.
	\end{equation}
\end{lemma}


Next, we exploit the Gaussian upper bound for the complex time heat kernel to calculate the kernel of multipliers. See e.g. Davies \cite[Theorem 3.4.8]{dav}, Carron-Coulhon-Ouhabaz \cite[Prop. 4.1]{cco}, Duong-Robinson \cite{dr96}, Coulhon-Duong \cite{cd00}. 
\begin{lemma}\label{cpxheat} Let $p(z,x,y)=e^{-zL}(x,y)$ for $z\in \mathbb{C}$ with ${\rm Re}\ z>0$.
	Then there exist $C>0$ and $c>0$ independent of $z$ such that for  all  $x,y\in M$ we have
	\begin{equation}
		|p(z,x,y)|\le C({\rm Re}\ z)^{-m/2}e^{-{\rm Re}(cd(x,y)^2/z)}.
	\end{equation}
\end{lemma}

Let $\psi\in C_0^\infty(\mathbb{R})$ and  $\psi_1(x)=\psi(x)e^{x}$. Then by the Fourier inversion formula we have
$$\psi(x)=\frac1{2\pi}\int_{\mathbb{R}}\hat\psi_1(s)e^{-(1-is)x}ds.$$
Then for $h>0$,
\[\psi(hL)=\frac1{2\pi}\int_{\mathbb{R}}\hat\psi_1(s)e^{-(1-is)hL}ds.\]
Applying Lemma \ref{cpxheat} to this formula, we obtain the following pointwise estimates, since $\hat \psi_1$ is a Schwartz function.
\begin{corr}\label{smoothker}For any $N>0$, there exists $C_N>0$ such that for all $\la\ge1$ and $x,y\in M$
	\begin{equation}\label{kerupp}
		|\psi(L/\la^2)(x,y)|\le C_N\la^m(1+\la d(x,y))^{-N}.
	\end{equation}
\end{corr}
The kernel estimate \eqref{kerupp} is known on compact manifolds without boundary, and it follows from the Hadamard parametrix for the wave equation, see e.g.  Sogge \cite[Theorem 4.3.1]{fio}. Nevertheless, this wave equation method seems hard to work on compact manifolds with boundary, since it is difficult to construct a precise parametrix for the wave equation near the boundary, see e.g. Seeley \cite{seeley1, seeley2}, Ivrii \cite{ivrii}, Melrose-Taylor \cite{MTbook}, Smith-Sogge \cite{ssjams,ssacta} and H\"ormander \cite{hor4}.

\section{Proof of Theorem \ref{LSthm}: the sufficiency part}
In this section, we  show that $\mathcal{A}$ is $L^p$-Logvinenko-Sereda if it is relatively dense. We shall modify the boundaryless argument in \cite[Proof of Prop. 5.11]{OCP13} in order to deal with  compact manifolds with boundary. Indeed, we need to handle the Dirichlet case and the Neumann case separately.  Moreover, we shall replace the $L^2$-orthogonality argument in \cite[Prop. 3.1]{OCP13} by Lemma \ref{Lplemma}, in order to handle $1<p<\infty$.

Let $\eps>0$ and $r>0$ be fixed constants to be determined later. Let $$f(z)=\sum_{\la_j\le \la}\alpha_je_j(z),$$
and its harmonic extension \begin{equation}\label{ht}
	h(z,t)=\sum_{\la_j\le \la}\alpha_je^{\la_jt}e_j(z).
\end{equation}
We choose $\psi\in C_0^\infty(\mathbb{R})$ such that 
\[\1_{[-1,1]}\le \psi\le \1_{[-2,2]}.\]Let $\varphi(s)=e^{ts}\psi(s/\la)$ with $|t|\le r/\la$. Clearly $h(z,t)=\varphi(\sqrt{-\Delta})f(z)$ and $\varphi\in S^0$ satisfies \eqref{pdo} with constants independent of $t$ and $\la$.
By Lemma \ref{Lplemma}, there exists $C>0$ such that for all $1<p<\infty$ and $|t|\le r/\la$
\begin{equation}\label{pdobd}
	\int_M|h(z,t)|^pdV(z)\le C\int_{M}|f(z)|^pdV(z).
\end{equation}

Let $N=M\times [-1,1]$. Then $f(z)=h(z,0)$ and $h$ is harmonic in $N$.  Given $f$ and its harmonic extension $h$, we define
\begin{equation}\label{defD}
	D=D_{\eps,r,f}=\Big\{z\in M: |f(z)|^p\ge \eps (\la/r)^{m+1}\int_{B(z,r/\la)}\int_{|t|\le r/\la}|h(\xi,t)|^pdV(\xi) dt \Big\}.
\end{equation}
Then by the definition of $D$ and \eqref{pdobd} we obtain
\begin{align*}
	\int_{M\setminus D} |f(z)|^pdz &\le \eps (\la/r)^{m+1}\int_{M\setminus D}\int_{B(z,r/\la)}\int_{|t|\le r/\la}|h(\xi,t)|^pdV(\xi) dtdV(z)\\
	&\ls \eps(\la/r)\int_{M}\int_{|t|\le r/\la}|h(\xi,t)|^pdV(\xi) dt\\
	&\ls \eps\int_M|f(z)|^pdV(z).
\end{align*}
We can choose $\eps$ small enough so that the contribution of $M\setminus D$ can be ignored. So it suffices to show that there exists $C>0$ such that for all $\la\ge1$ and $f\in E_\la$,
\begin{equation}\label{ineq0}
	\int_D|f|^pdV\le C \int_{A_\la}|f|^pdV.
\end{equation}
We only need to prove that there exists $C>0$ such that for all  $\la\ge1$, $f\in E_\la$ and $w\in D$,
\begin{equation}\label{ineq1}
	|f(w)|^p\le C (\la/r)^{m}\int_{A_\la\cap B(w,r/\la)}|f|^pdV.
\end{equation}

We recall the mean-value inequality  and the gradient estimate  for harmonic functions in  \cite{OCP13}. Thanks to Gilbarg-Trudinger \cite[Theorem 8.17]{GT98} for weak solutions in $W^{1,2}$, for all sufficiently small $r>0$ and $1<p<\infty$, the mean-value inequality for the harmonic function $h$ 
\begin{equation}\label{mv}
	|h(z,t)|^p\le C (\la/r)^{m+1}\int_{B(z,r/\la)}\int_{|s-t|\le r/\la}|h(w,s)|^pdV(w)ds
\end{equation}
is valid when the boundary is empty or the boundary condition is either Dirichlet or Neumann. Note that if $h$ is harmonic in the Euclidean sense, then \eqref{mv} just follows from the classical mean value equality and H\"older inequality. On the other hand, for the harmonic function $h$ on $B(w,2\rho)$, the gradient estimate 
\begin{equation}\label{gradest}
	\sup_{B(w,\rho)}|\nabla h|\le C \rho^{-1}\sup_{B(w,2\rho)}|h|
\end{equation}is known when the ball $B(w,2\rho)$  does not intersect the boundary of $M$, see e.g. Schoen-Yau \cite[Corollary 3.2]{sy94}, Li \cite[Theorem 6.1]{li12}. Nevertheless, to our best knowledge, the gradient estimate \eqref{gradest} seems unknown when the ball $B(w,2\rho)$ intersects the boundary. The  estimate \eqref{gradest} will be used to prove the equicontinuity in Step 3 of Section \ref{bdless}.

In Section \ref{bdless}, we first give the proof  for the boundaryless case  in order to describe the modifications to be made in the case with boundary later. The argument is similar to the proof of \cite[Prop. 5.11]{OCP13}.

In Section \ref{dir}, we handle the Dirichlet case. We show that the contribution of a small neighborhood of the boundary can be ignored, see Lemma \ref{dirlem}. So we only need to consider those $w$ outside this neighborhood, and then the ball $B(w,r/\la)$ does not intersect the boundary if we choose $r$ small enough. Thus, we can use the gradient estimate \eqref{gradest} and the argument of the boundaryless case to finish the proof.

In Section \ref{neu}, we handle the Neumann case. We can establish the gradient estimate of the form \eqref{gradest} when the ball $B(w,2\rho)$ intersects the boundary under Neumann boundary condition by using the strategy in Li \cite{li12}, Chen \cite{chen90},  Wang \cite{wang97}. So the proof for the boundaryless case can still work in the Neumann case.

\subsection{Boundaryless case}\label{bdless} We assume \eqref{ineq1} is not true in order to construct functions that satisfy the opposite inequality. Then we will parametrize these functions and prove that their limit is  analytic and  vanishes on a subset of positive measure. This will lead to a contradiction.

Suppose \eqref{ineq1} is not true. Then for each $n\ge1$, there exist $\la_n\ge1$, $f_n\in E_{\la_n}$ and $w_n\in D_n$ such that
\begin{equation}\label{ass}
	|f_n(w_n)|^p> n(\la_n/r)^{m}\int_{A_{\la_n}\cap B(w_n,r/\la_n)}|f_n|^pdV.
\end{equation}
Here $D_n$ is defined as in \eqref{defD} with $f$ replaced by $f_n$.

\noindent \textbf{Step 1. Scaling.} By the compactness of the manifold, we have the exponential map $$\exp_w:\mathbb{B}(0,\rho_0)\subset\mathbb{R}^m\to B(w,\rho_0)\subset M$$ for some $\rho_0=\rho_0(M)>0$, such that $\exp_w(0)=w$ and $(B(w,\rho_0),\exp_w^{-1})$ is a normal coordinate chart and the metric $g_{ij}(w)=\delta_{ij}$. Now we rescale the metric. Let $\exp_n(z)=\exp_{w_n}(rz/\la_n)$ and $$d\mu_n(z)=\sqrt{|g|(\exp_n(z))}dz.$$ Let $$F_n(z)=c_nf_n(\exp_n(z))$$ and its harmonic extension $$H_n(z,t)=c_nh_n(\exp_n(z),rt/\la_n),$$ where $c_n$ is a constant such that
\[\int_{\mathbb{B}(0,1)}\int_{|t|<1}|H_n(z,t)|^pd\mu_n(z)dt=1.\]
By scaling, it is equivalent to
\begin{equation}\label{heq}
	|c_n|^p(\la_n/r)^{m+1}\int_{B(w_n,r/\la_n)}\int_{|t|<r/\la_n}|h_n(w,t)|^pdV(w)dt=1.
\end{equation}
\noindent \textbf{Step 2. Uniform boundedness.} By the definition of $D_n$ and \eqref{heq} we have
\begin{equation}\label{low0}
	|F_n(0)|^p=|c_n|^p|f_n(w_n)|^p\ge \eps.
\end{equation}
By the mean value inequality \eqref{mv} and \eqref{heq} we get
\begin{equation}\label{upp0}
	|F_n(0)|^p=|c_n|^p|h_n(w_n,0)|^p\ls1.
\end{equation}
Similarly, if $z\in \mathbb{B}(0,\frac12)$ and $|s|<\frac12$, then for all $w=\exp_n(z)\in B(w_n,r/(2\la_n))$ and $t=sr/\la_n$ we have $B(w,r/(2\la_n))\subset B(w_n,r/\la_n)$ and then  \eqref{mv} and \eqref{heq} imply
\begin{equation}\label{Hupp}
	|H_n(z,s)|^p=|c_n|^p|h_n(w,t)|^p\ls 1.
\end{equation}
So working with $\frac12$ instead of 1 we have proved the sequence $H_n$ is uniformly bounded.

Furthermore, if $B_n\subset \mathbb{B}(0,1)$ is such that $\exp_n(B_n)=A_{\la_n}\cap B(w_n,r/\la_n)$, then the assumption \eqref{ass} and \eqref{upp0} imply
\begin{equation}\label{L2small}
	\int_{B_n}|F_n|^pd\mu_n\ls \frac1n.
\end{equation}

\noindent \textbf{Step 3. Equicontinuity.} Let $\eps_1>0$. For $w\in B(w_n,r/(4\la_n))$, $\tilde w\in B(w,\eps_1r/\la_n)$, $|t|<r/(4\la_n)$ and $|s-t|<\eps_1r/\la_n$, we have $|(w,t)-(\tilde w,\tilde t)|\le\eps_1r/\la_n$. If $\eps_1$ is small enough, then by the gradient estimate \eqref{gradest} and \eqref{Hupp} we get
\begin{align*}
	|c_n||h_n(w,t)-h_n(\tilde w,\tilde t)|&\le |c_n|(\eps_1r/\la_n)\sup_{\omega\in B(w_n,r/(2\la_n),|t|<r/(2\la_n)}|\nabla h_n(w,t)|\\
	&\le C|c_n|(\eps_1r/\la_n)(\la_n/r)\sup_{\omega\in B(w_n,r/\la_n),|t|<r/\la_n} |h_n(w,t)|\\
	&\ls \eps_1.
\end{align*}
By scaling, for all $\eps>0$, there exists $\eps_1>0$ such that for all $n$
\[|H_n(z,s)-H_n(\tilde z,\tilde s)|<\eps,\]
whenever $|z-\tilde z|+|s-\tilde s|<\eps_1$, $z\in \mathbb{B}(0,\frac14)$ and $|t|< \frac14$. So the sequence $H_n$ is equicontinuous.

\noindent \textbf{Step 4. Analyticity.}
Since the sequence $H_n$ is uniformly bounded and equicontinuous on $\mathbb{B}(0,1)\times(-1,1)$, by Ascoli-Arzela's theorem, there is a subsequence (denoted as the sequence itself) $H_n\to H$ uniformly on any compact subset of $\mathbb{B}(0,1)\times(-1,1)$. Note that $H_n$ is harmonic with respect to the rescaled metric $g_n(z)=g(\exp_n(z))$, which has a subsequence (denoted as the sequence itself) uniformly converges to the Euclidean metric. We can see that $H$ is harmonic  in the Euclidean sense by \cite[Lemma 5.9]{OCP13}. This implies $H$ is real analytic.

\noindent \textbf{Step 5. Contradiction.}
 Recall  $B_n\subset \mathbb{B}(0,1)$ satisfies $\exp_n(B_n)=A_{\la_n}\cap B(w_n,r/\la_n)$. Since the sequence $A_{\la_n}$ is relatively dense, there exists $\rho>0$ such that for all $n$
\begin{equation}\label{rd}
	\mu_n(B_n)\ge \rho.
\end{equation} Let $\tau_n$ be such that $d\tau_n=\1_{B_n}d\mu_n$. Then there is a subsequence (denoted as the sequence itself) converging weakly* to $\tau$, which is not identically zero by \eqref{rd}. If $\overline{\mathbb{B}(a,s)}\subset\mathbb{B}(0,1)$, then $\tau_n(\overline{\mathbb{B}(a,s)})\le\mu_n(\overline{\mathbb{B}(a,s)})\ls s^m$. Thus, $\tau(\overline{\mathbb{B}(a,s)})\ls s^m$ and we have
\[0<\tau(\supp\tau)\ls \mathcal{H}^m(\supp\tau),\]
which implies that the Hausdorff dimension of $\supp\tau$ is $m$. On the other hand, let $F(z)=H(z,0)$. Then by \eqref{L2small}, we get
\[\int_{\mathbb{B}(0,1)}|F|^pd\tau=0.\]
This implies that $\supp\tau \subset F^{-1}(0)$. But $F^{-1}(0)$ is of Hausdorff dimension $m-1$, which is a contradiction.

\subsection{Dirichlet case}\label{dir}
For $\delta>0$, let $M_\delta=\{x\in M:d(x,\partial M)<\delta\}$.
\begin{lemma}\label{dirlem}
Let $1\le p<\infty$. There exist $\delta_0>0$ and $C>0$ such that	for all $0<\delta<\delta_0$, $\la\ge1$ and $f\in E_\la$, we have
\begin{equation}\label{nearbd}
	\|f\|_{L^p(M_{\delta/\la})}\le C\delta\|f\|_{L^p(M)}.
\end{equation}
\end{lemma}
By the lemma with $\delta$ small enough, we can ignore the contribution of $M_{\delta/\la}$ and only need to prove \eqref{ineq0} with $D$ replaced by $D\setminus M_{\delta/\la}$. For $\omega\in D\setminus M_{\delta/\la}$,  the  ball $B(w,r/\la)$ does not intersect the boundary of $M$ if we choose $r$ small enough.  So we can use the argument in Section \ref{bdless} to prove the Dirichlet case. 

\noindent\textbf{Proof of Lemma \ref{dirlem}.} Choose $\delta_0$ small enough so that we can use Fermi coordinates $z=(s,x)$ with $s\in [0,\delta_0/\la]$ and $x\in \partial M$. Under the Fermi coordinates, the metric is in the following form
\[ g = ds^2+\tilde{g}_{ij}(s)dx^idx^j,\]
where $\tilde{g}(s)=\tilde{g}_{ij}(s)dx^idx^j$ is the pullback of the induced metric on $$\Sigma_s = \{x\in M:d(x,\partial M) = s\}.$$ Let $d\sigma(s)$ be the volume form of $\tilde{g}(s)$, then $d\sigma(s)\approx d\sigma(0)=d\sigma$.
Using the Dirichlet boundary condition, we have
\[f(s,x)=\int_0^s\partial_\tau f(\tau,x)d\tau.\]
So we obtain for $1\le p<\infty$
\begin{align*}
	\int_{M_{\delta/\la}}|f(z)|^p&\approx \int_0^{\delta/\la}\int_{\partial M}|f(s,x)|^pd\sigma(x) ds\\
	&= \int_0^{\delta/\la}\int_{\partial M}\Big|\int_0^s\partial_\tau f(\tau,x)d\tau\Big|^pd\sigma(x) ds\\
	&\le \int_0^{\delta/\la}\int_{\partial M}s^{p-1}\int_0^s|\partial_\tau f(\tau,x)|^pd\tau d\sigma(x) ds\\
	&\le \int_0^{\delta/\la}\int_{\partial M}((\delta/\la)^p-\tau^p)|\partial_\tau f(\tau,x)|^pd\tau d\sigma(x) \\
	&\le (\delta/\la)^p\int_0^{\delta/\la}\int_{\partial M}|\partial_\tau f(\tau,x)|^pd\tau d\sigma(x) \\
	&\le (\delta/\la)^p\|\nabla f\|_{L^p(M)}^p
\end{align*}
 Since $f\in E_\la$,  by the Bernstein inequality \eqref{bern} we get 
\[\|\nabla f\|_{L^p(M)}\ls \la\|f\|_{L^p(M)}.\]
So we complete the proof of Lemma \ref{dirlem}.
\subsection{Neumann Case}\label{neu}
To exploit  the argument in Section \ref{bdless} to handle the Neumann case, we only need to establish the gradient estimate of the form \eqref{gradest}  under Neumann boundary condition.  Our proof is inspired by the book of Li \cite[Proof of Theorem 6.1]{li12}, Chen \cite{chen90}, and Wang \cite{wang97}.

We first prove the gradient estimate 
$$
\sup_{B(x_c,\rho)} |\nabla h| \ls \rho^{-1} \sup_{B(x_c,2\rho)} |h|  
$$
for sufficiently small $\rho$ under the condition $d(x_c, \partial M) \ge \frac{\rho}{2}$. For the remaining cases, covering $B(x_c,\rho)$ by finitely many balls $B(x_i,r_i)$ with $d(x_i, \partial M) \ge r_i$ yields
\begin{equation}\label{grad}
	\sup_{B(x_c,\rho)} |\nabla h| \ls \rho^{-1} \sup_{B(x_c,K\rho)} |h|  
\end{equation}
for some constant $K=K(M)$.  The estimate \eqref{grad} is slightly weaker than \eqref{gradest}, but it is sufficient for our purpose.

 Let $A=2\sup_{B(x_c,2\rho)}|h|$. Then $u=h+A$ is a positive harmonic function  in $B(x_c,2\rho)$. Let $v=\log u$ and $Q=|\nabla v|^2$. It suffices to prove $Q(x)\ls \rho^{-2}$ for all $x\in B(x_c,\rho)$ with $\rho$ sufficiently small.

Let $r(x)$ be the distance between $x$ and $x_c$. Let us choose $\phi=\phi(r(x))$ as in  the book of Li \cite[p. 60]{li12}. It is a nonnegative function with the property that 
\[\phi=1\ \text{on}\ B(x_c,\rho),\]
\[\phi=0\ \text{on}\ M\setminus B(x_c,2\rho),\]
\[-C\rho^{-1}\sqrt\phi\le \phi'\le0\ \text{and}\ |\phi''|\le C\rho^{-2}\ \text{on}\ B(x_c,2\rho)\setminus B(x_c,\rho).\]
Let $d(x)$ be  the distance between $x$ and the boundary $\partial M$ and let $$F(x)=(1+\psi(d(x)))^2\phi(r(x))(Q(x)+1),$$ where $\psi$ is a nonnegative smooth function with  $\psi(0)=0$ and $\psi'(0)=H>0$, if the second fundamental form II of $\partial M$ is bounded below by $-H$.  We first prove by contradiction that $F$ cannot achieve the maximum at any boundary point, and then we estimate the maximum in the interior.

Suppose $F$ achieves the maximum at $q\in \partial M$. Choose a Fermi coordinate at $q$ with $x=(x',x_m)$, where $x_m$ is inner normal to the boundary and $x'$ is a normal coordinate on $\partial M$ at $q$. Then we have $d(x)=x_m\ge0$ and
\begin{align*}
	\frac{\partial F}{\partial x_m}&=(1+\psi(x_m))^2\phi'(r(x))\frac{\partial r(x)}{\partial x_m}(Q(x)+1)\\
	&\ \ +2(1+\psi(x_m))\psi'(x_m)\phi(r(x))(Q(x)+1)\\
	&\ \ +(1+\psi(x_m))^2\phi(r(x))\frac{\partial Q(x)}{\partial x_m}.
	\end{align*}
We want to show $\frac{\partial F}{\partial x_m}(q) > 0$, which leads to a contradiction.  
For sufficiently small $\rho$, since $d(x_c,\partial M) \ge \frac{\rho}{2}$, we have $\frac{\partial r(x)}{\partial x_m} \leq 0$. Noting that $\phi' \leq 0$, the first term is always nonnegative. To handle the remaining two terms, we denote $u_j = \partial_{x_j} u$ and $u_{jm} = \partial_{x_j}\partial_{x_m} u$.
By the Neumann boundary condition $u_m|_{x_m=0}=0$, we have
\begin{align*}
	\frac{\partial Q(x)}{\partial x_m}\Big|_{x=q}&=\frac{2\sum_j u_ju_{jm}}{u^2}-\frac{2|\nabla u|^2u_m}{u^3}\Big|_{x=q}\\
	&=\frac{2\sum_{j<m} u_ju_{jm}}{u^2}\Big|_{x=q}\\
	&=\frac{2\text{II}(\nabla u,\nabla u)}{u^2}\Big|_{x=q}\\
	&\ge \frac{-2H|\nabla u(q)|^2}{u(q)^2},
\end{align*}
since the second fundamental form II of $\partial M$ is bounded below by $-H$. So we only need to show
	\[\psi'(0)(\frac{|\nabla u(q)|^2}{u(q)^2}+1)-(1+\psi(0))\frac{H|\nabla u(q)|^2}{u(q)^2}>0.\]
It is ensured by our choice of  $\psi$ with $\psi(0)=0$ and $\psi'(0)=H>0$.
	
Next, we assume $F$ achieves the maximum at $x_0\in M$. Then \begin{equation}\label{critical}
	\Delta F(x_0)\le 0,\  \text{and} \ \nabla F(x_0)=0.
\end{equation} Let $\varphi(x)=(1+\psi(d(x)))^2\phi(r(x))$. Since $d(x)=x_m$ is smooth, $\varphi(x)$ and its derivatives satisfy similar bounds as $\phi$. We may assume the Ricci curvature on $M$ is bounded below by $-(m-1)R$ $(R\ge0)$, since $M$ the compact. Then we repeat  the standard argument as in the book of Li \cite[Proof of Theorem 6.1]{li12} to calculate $\Delta F$ and simplify the expression evaluated at $x_0$ by the condition \eqref{critical}. As in \cite[(6.2)]{li12}, we get
\begin{align*}
	0\ge (m-1)(\Delta\varphi)F-\frac{3m-4}2|\nabla\varphi|^2\varphi^{-1}F-2(m-1)^2R\varphi F-2(m-2)|\nabla\varphi|\varphi^{-1/2}F^{3/2}+2F^2.
\end{align*}
By the Laplacian comparison theorem \cite[Theorem 4.1]{li12}, we have $\Delta\varphi\gs -\rho^{-2}$. Since  $|\nabla\varphi|^2\varphi^{-1}\ls \rho^{-2}$ and $\varphi\ls1$, we obtain
\begin{align*}
	0&\ge -C_1\rho^{-2}F-C_2 F-C_3\rho^{-1}F^{3/2}+2F^2\\
	&\gs -\rho^{-2}F+F^2
\end{align*}
So  $F(x)\le F(x_0)\ls \rho^{-2}$. Therefore, we obtain $Q(x)\ls \rho^{-2}$ for all $x\in B(x_c,\rho)$, as $\varphi(x)\approx 1$ in $B(x_c,\rho)$.
	\section{Proof of Theorem \ref{LSthm}: the necessity part}
	In this section, we prove that $\mathcal{A}$ is relatively dense if it is $L^p$-Logvinenko-Sereda. By the heuristic in the introduction, we shall construct a test function that behaves like the indicator function of the ball of radius $1/\la$.
	
	For any  fixed $r>0$ and $C>0$, if   for all $\la\ge1$ and $y\in M$ with $d(y,\partial M)\ge C/\la$,
\begin{equation}\label{rdy}
	\text{vol}(A_\la\cap B(y,r/\la))\ge \rho \la^{-m},
\end{equation}
	then there exists  $c>0$ such that for all $\la\ge1$ and $z\in M$
	\[\text{vol}(A_\la\cap B(z,(r+C)/\la))\ge c\rho\la^{-m},\]since any ball $B(z,(r+C)/\la)$ contains a ball $B(y,r/\la)$ with $d(y,\partial M)\ge C/\la$.
	So to prove $\mathcal{A}$ is relatively dense, we may assume the centers of the balls are $C/\la$ away from the boundary. We will choose sufficiently large $r$ and $C$ later to prove \eqref{rdy}.
	
	Let $\psi\in C_0^\infty(\mathbb{R})$ with $\1_{[-\frac{1}{2},\frac{1}{2}]}\le \psi\le \1_{[-1,1]}$. Given  $L$ as in Corollary \ref{longheat} such that  \eqref{heat} and \eqref{heatgrad} hold  for all $t>0$ and $x,y\in M$. We define
	\[f_{\la,y}(x)=\psi(L/\la^2)(x,y).\]
	By the kernel estimate in Corollary \ref{smoothker}, we have
	\[|f_{\la,y}(x)|\ls \la^m(1+\la d(x,y))^{-N}.\]
	Since $d(y,\partial M)>C/\la$, by Corollary \ref{awaylow} we get $|f_{\la,y}(y)|\approx \la^m.$
	Together with the Bernstein inequality \eqref{bern1}  and the mean value theorem, this implies that there exists $c>0$ independent of $y$ such that for all $x\in B(y,c/\la)$,
	\[|f_{\la,y}(x)|\approx \la^m.\]
	Thus, the function $f_{\la,y}$ behaves like the indicator function of the ball of radius $1/\la$, which corresponds to the heuristic in the introduction.  For $1\le p<\infty$ we have
	\[\int_M |f_{\la,y}|^p\ge \int_{B(y,c/\la)}|f_{\la,y}|^p\approx \la^{mp-m}.\]
	On the other hand, since $\mathcal{A}$ is $L^p$-Logvinenko-Sereda, we have for $N>m$
	\begin{align*}
		\int_M|f_{\la,y}|^p&\le C \int_{A_\la}|f_{\la,y}|^p\le C\int_{A_\la\cap B(y,r/\la)}|f_{\la,y}|^p+C\int_{M\setminus B(y,r/\la)}|f_{\la,y}|^p\\
		&\le C_1\la^{mp}\text{vol}(A_\la\cap B(y,r/\la))+C_2\la^{mp}\int_{M\setminus B(y,r/\la)}(1+\la d(x,y))^{-N}dV(x)\\
		&\le C_1\la^{mp}\text{vol}(A_\la\cap B(y,r/\la))+C_3r^{-N+m}\la^{mp-m}.
	\end{align*}
	Therefore, choosing $r$ large enough, we obtain 
	\[\text{vol}(A_\la\cap B(y,r/\la))\gs \la^{-m}.\]
	This proves \eqref{rdy}, which implies that $\mathcal{A}$ is relatively dense.
	\section{Proof of Theorem \ref{carthm}}
	Let $\mu=\{\mu_\la\}_\la$ be a sequence of measures on $M$. 	\subsection{Sufficiency part}
	Suppose there exist $r>0$ and $C>0$ such that for all $\la\ge1$ and $z\in M$, we have $$\mu_\la(B(z,r/\la))\le C\la^{-m}.$$
	 For any $f\in E_\la$, by the mean-value inequality \eqref{mv} and the $L^p$ bound \eqref{pdobd} for the harmonic extension $h$ in \eqref{ht}, we have for $1<p<\infty$
	\begin{align*}
		\int_M|f|^pd\mu_\la&\le C (\la/r)^{m+1}\int_M\int_{B(z,r/\la)}\int_{|s|\le r/\la}|h(w,s)|^pdV(w)dsd\mu_\la(z)\\
		&=C(\la/r)^{m+1}\int_M\int_{|s|\le r/\la}|h(w,s)|^p\mu_\la(B(w,r/\la))dV(w)ds\\
		&\ls (\la/r)\int_M\int_{|s|\le r/\la}|h(w,s)|^pdV(w)ds\\
		&\ls \int_M|f|^pdV.
	\end{align*}
	Thus,  $\mu$ is $L^p$-Carleson.
	\subsection{Necessity part}
 Suppose that the boundary is empty or the  boundary condition is Neumann. We assume $\mu$ is $L^p$-Carleson. We need to prove that  there exist $C>0$ and $r>0$ such that for all $\la\ge1$ and $\xi\in M$,
	\begin{equation}\label{lpcar}
		\mu_\la(B(\xi,r/\la))\le C\la^{-m}.
	\end{equation}
We choose $\psi\in C_0^\infty(\mathbb{R})$ such that 
\[\1_{[-\frac12,\frac12]}\le \psi\le \1_{[-1,1]}.\]	Let $F_{\la,\xi}(x)=\psi(L/\la^2)(\xi,x)$ where $L$ is chosen as in Corollary \ref{longheat} such that  \eqref{heat} and \eqref{heatgrad} hold  for all $t>0$ and $x,y\in M$. Then $F_{\la,\xi}\in E_{\la}$. By the kernel estimate in \eqref{kerupp}, we have
\[|F_{\la,\xi}(x)|\ls \la^m(1+\la d(\xi,x))^{-N},\ \ \forall N.\]So for $1\le p\le \infty$ we get
\begin{equation}\label{Lpbd}
	\|F_{\la,\xi}\|_{L^p(M)}\ls\la^{m(1-\frac1p)}.
\end{equation}
Moreover, by using the two-sided estimate in \eqref{twosideest}, we obtain $F_{\la,\xi}(\xi)\approx \la^m$. Together with the Bernstein inequality \eqref{bern1}  and the mean value theorem, this implies that there exists $r>0$ independent of $\xi$  such that for all $x\in B(\xi,r/\la)$ $$|F_{\la,\xi}(x)|\approx \la^m.$$ Thus, the function $F_{\la,\xi}$ behaves like the indicator function of the ball of radius $1/\la$, which corresponds to the heuristic in the introduction.

So we get for $1\le p<\infty$
\begin{align*}
 \int_M|F_{\la,\xi}|^pd\mu_\la\ge \int_{B(\xi,r/\la)}|F_{\la,\xi}|^pd\mu_\la\gs\la^{mp}\mu_\la(B(\xi,r/\la)).
\end{align*}
Together with \eqref{Lpbd} and the assumption that $\mu$ is $L^p$-Carleson, this implies \eqref{lpcar}.

\subsection{Counterexample for Dirichlet boundary condition} Suppose that the boundary condition is Dirichlet. We construct an example to show that the relatively sparse condition is not necessary for $\mu$ to be $L^p$-Carleson. Fix $z\in \partial M$. Let 
\[d\mu_\la=dV+\delta_{z},\] where $\delta_z$ is the Dirac measure supported at $z$. For any $f=\sum_{\la_j\le \la}\alpha_je_j\in E_\la$,  we have $f(z)=0$ by the Dirichlet boundary condition. Then for $1\le p<\infty$,
\begin{align*}
	\int_M|f|^pd\mu_\la=\int_M|f|^pdV.
\end{align*}
So $\mu$ is $L^p$-Carleson. On the other hand, when $z\in B(\xi,1/\la)$, we have
\[\mu_\la(B(\xi,1/\la))= {\rm vol}(B(\xi,1/\la))+1.\]
But ${\rm vol}(B(\xi,1/\la))\approx \la^{-m}\ll1$ for sufficiently large $\la$. So $\mu$ is not relatively sparse.

\section{Logvinenko-Sereda sets for eigenfunctions}
\subsection{Proof of Proposition \ref{smallpls}} For $1\le p< \frac{2m}{m-1}$, we need to show the relatively dense condition  \eqref{RDdef} is not necessary for $\mathcal{A}$ to be  $L^p$-Logvinenko-Sereda  for eigenfunctions on $M$. We fix any $z\in M$. For any ball $B(z,r)$  and any $L^2$-normalized eigenfunction $e_\la$ we have $\|e_\la\|_{L^\infty(M)}\ls \la^{\frac{m-1}2}$ and then for $1\le p< \infty$
$$\int_{B(z,r)}|e_\la|^pdV\ls \la^{\frac{m-1}2p}r^{m}.$$ See e.g. Sogge \cite{fio}.
On the other hand, by H\"older inequality we get
\[\int_M|e_\la|^pdV\gs \begin{cases}1,\ \ \ \ \ \ \ \ \ \ \ \ \ \ 2\le p<\infty,\\
	\la^{-\frac{m-1}2(2-p)},\ \ 1\le p<2.
\end{cases}\]
Let 
\[R_\la=\begin{cases}
	c(\la)\la^{-\frac{m-1}{2m}p},\ \ 2\le p<\infty\\
	c(\la)\la^{-\frac{m-1}m},\ \ \ 1\le p<2,
\end{cases}\]
where $c(\la)=(\log\la)^{-1}$.
Let $$A_\la=M\setminus B(z, R_\la).$$ Then for any $\eps>0$, we have for sufficiently large $\la$
\begin{equation}\label{Lpball}
	\int_{B(z,R_\la)}|e_\la|^pdV\le \eps \int_{M}|e_\la|^pdV,
\end{equation}
which implies that
\[	\int_{A_\la}|e_\la|^pdV\approx \int_{M}|e_\la|^pdV.\]
So $\mathcal{A}=\{A_\la\}_\la$ is $L^p$-Logvinenko-Sereda for eigenfunctions. On the other hand, we have $R_\la\gg \la^{-1}$ for $1\le p<\frac{2m}{m-1}$ and sufficiently large $\la$. So $\mathcal{A}$ is not relatively dense, since $A_\la\cap B(z,r/\la)=\emptyset$ for any $r>0$ and sufficiently large $\la$. 

To handle the kink point $p=\frac{2m}{m-1}$, we recall that $\|e_\la\|_{L^\infty(M)}=o( \la^{\frac{m-1}2})$ for a generic metric on any compact manifold. See Sogge-Zelditch \cite[Theorem 1.4]{sz01}. So we can pick $R_\la=c(\la)\la^{-1}$ with $c(\la)\nearrow \infty$ as $\la\to\infty$ such that for any $\eps>0$ and sufficiently large $\la$
\[\int_{B(z,R_\la)}|e_\la|^pdV\le \|e_\la\|_{L^\infty(M)}^p\text{vol}(B(z,R_\la))\le \eps\int_M|e_\la|^pdV.\]
As before, if we define $A_\la=M\setminus B(z, R_\la)$, then
\[	\int_{A_\la}|e_\la|^pdV\approx \int_{M}|e_\la|^pdV.\]
So $\mathcal{A}=\{A_\la\}_\la$ is $L^p$-Logvinenko-Sereda for eigenfunctions. But it is not relatively dense since $R_\la\gg \la^{-1}$ and then $A_\la\cap B(z,r/\la)=\emptyset$ for any $r>0$ and sufficiently large $\la$. 

\subsection{Proof of Proposition \ref{smallplssphere}} We just need to slightly modify the proof of Proposition \ref{smallpls}. Indeed, we can replace the definition of $A_\la$ by $M\setminus \cup_\pm B(\pm z,R_\la)$ on the sphere $S^m$ and show that $$A_\la\cap \cup_\pm B(\pm z,r/\la)=\emptyset$$ for any $r>0$ and sufficiently large $\la$. This shows that $\mathcal{A}$ is not relatively dense.

\subsection{Proof of Theorem \ref{largepls}}
Let $M=S^m$ be the standard sphere. For $A_\la\subset S^m$, let $$\tilde A_\la=A_\la\cup (-A_\la).$$
By definition, it is easy to see that $\mathcal{A}=\{A_\la\}_\la$ is symmetric relatively dense if and only if $ \tilde{\mathcal{A}}=\{\tilde A_\la\}_\la$ is relatively dense. 

Suppose $\mathcal{A}=\{A_\la\}_\la$ is symmetric relatively dense. By Theorem \ref{LSthm},   $ \tilde{\mathcal{A}}$ is $L^p$-Logvinenko-Sereda. The eigenfunctions on $S^m$ are the restriction to $S^m$ of homogeneous harmonic polynomials in $\mathbb{R}^{m+1}$, so they must be either symmetric or antisymmetric. So
\[2\int_{A_\la}|e_\la|^pdV=\int_{A_\la}|e_\la|dV+\int_{-A_\la}|e_\la|^pdV\ge \int_{\tilde A_\la}|e_\la|^pdV\gs \int_M|e_\la|^pdV.\]
So $\mathcal{A}=\{A_\la\}_\la$ is $L^p$-Logvinenko-Sereda for eigenfunctions on $S^m$.

On the other hand, we can use zonal functions as test functions to show that the symmetric relatively dense condition is necessary. Let \begin{equation}\label{zonal}
	Z_{\la,\xi}(x)=\la^{-\frac{m-1}2}\sum_{\la_j=\la}e_j(x)\overline{e_j(\xi)}.
\end{equation}
We have $\|Z_{\la,\xi}\|_{L^p(M)}\approx 1$ for $1\le p<\frac{2m}{m-1}$ and  $\|Z_{\la,\xi}\|_{L^p(M)}\approx \la^{\frac{m-1}2-\frac mp}$ for $\frac{2m}{m-1}<p\le\infty$. At the kink point $p=\frac{2m}{m-1}$, we have $\|Z_{\la,\xi}\|_{L^p(M)}\approx (\log\la)^{\frac1p}$.  Moreover, $|Z_{\la,\xi}(x)|\approx \la^{\frac{m-1}2}$  for $x\in B(\pm\xi,1/\la)$ and we have $$|Z_{\la,\xi}(x)|\le C\la^{m-1}(1+\la d(x,\xi))^{-\frac{m-1}2}$$ whenever $1/\la\le d(x,\xi)\le3\pi/4$. See e.g. Sogge \cite{soggeprob}.

Suppose  $\mathcal{A}$ is $L^p$-Logvinenko-Sereda for eigenfunctions on $S^m$. We fix any $\xi\in S^m$. Then we have for $p>\frac{2m}{m-1}$
\begin{align*}
	\la^{\frac{m-1}2p- m}&\approx \int_M|Z_{\la,\xi}|^pdV\le C \int_{A_\la}|Z_{\la,\xi}|^pdV\\
	&\le C\int_{ A_\la\cap B(\xi,r/\la)}|Z_{\la,\xi}|^p+C\int_{ A_\la\cap B(-\xi,r/\la)}|Z_{\la,\xi}|^p+C\int_{S^m\setminus \cup_{\pm}B(\pm \xi,r/\la)}|Z_{\la,\xi}|^p\\
	&\le C_1\la^{\frac{m-1}2p}(\text{vol}( A_\la\cap B(\xi,r/\la))+\text{vol}( A_\la\cap B(-\xi,r/\la)))\\
	&\quad\quad+C_2\la^{(m-1)p}\int_{r/\la\le d(x,\xi)\le 3\pi/4}(1+\la d(x,\xi))^{-\frac{m-1}2p}dV(x)\\
	&\le C_1\la^{\frac{m-1}2p}(\text{vol}( A_\la\cap B(\xi,r/\la))+\text{vol}( A_\la\cap B(-\xi,r/\la)))+C_3r^{m-\frac{m-1}2p}\la^{\frac{m-1}2p-m}.
\end{align*}
Therefore, choosing $r$ large enough, we obtain 
\[\text{vol}( A_\la\cap B(\xi,r/\la))+\text{vol}( A_\la\cap B(-\xi,r/\la))\gs \la^{-m}.\]
So $\mathcal{A}$ is symmetric relatively dense.

\section{Carleson measures for eigenfunctions}\label{careigen}

\subsection{Proof of Proposition \ref{smallp}}For $1\le p< \frac{2m}{m-1}$, we need to show the relatively sparse condition is not necessary for $\mu$ to be $L^p$-Carleson for eigenfunctions on the any compact manifold $M$. We fix any $\xi\in M$ and we define
\begin{equation}\label{mumeas}
	d\mu_\la=c(\la)\1_{B(\xi,1/\la)}dV,
\end{equation}
where $c(\la)=\log\la$.
Clearly, the relatively sparse condition is not satisfied since $c(\la)\nearrow \infty$  as $\la\to\infty$. For any $L^2$-normalized eigenfunction $e_\la$ on $M$, we have $\|e_\la\|_{L^\infty(M)}\ls \la^{\frac{m-1}2}$.
See e.g. Sogge \cite{fio}. Then for $1\le p< \infty$
\[\int_M |e_\la|^pd\mu_\la\ls c(\la) \la^{(\frac{m-1}2-\frac mp)p}.\]
On the other hand, by H\"older inequality we get
\[\int_M|e_\la|^pdV\gs \begin{cases}1,\ \ \ \ \ \ \ \ \ \ \ \ \ \ 2\le p<\infty,\\
	\la^{-\frac{m-1}2(2-p)},\ \ 1\le p<2.
\end{cases}\]
Thus, for any $1\le p<\frac{2m}{m-1}$ we have
\[\Big(\frac{m-1}2-\frac mp\Big)p<\min\Big\{0,-\frac{m-1}2(2-p)\Big\}.\]
Then for $1\le p<\frac{2m}{m-1}$, we have
\[\int_M |e_\la|^pd\mu_\la\le C\int_M|e_\la|^pdV.\]
So $\mu$ is $L^p$-Carleson for eigenfunctions on $M$, though it is not relatively sparse. 

For the kink point $p=\frac{2m}{m-1}$, we recall that $\|e_\la\|_{L^\infty(M)}=o( \la^{\frac{m-1}2})$ for a generic metric on any compact manifold. See Sogge-Zelditch \cite[Theorem 1.4]{sz01}. Thus, for $\mu_\la$ in \eqref{mumeas}, there exists a function $c(\la)\nearrow \infty$ as $\la\to\infty$ such that
\[\int_M |e_\la|^pd\mu_\la\le c(\la)\|e_\la\|_{L^\infty(M)}^p\text{vol}(B(\xi,1/\la))\ls 1\ls \int_M|e_\la|^pdV.\]
So $\mu$ is $L^p$-Carleson for eigenfunctions on  $M$ with a generic metric, though it is not relatively sparse.

\subsection{Proof of Theorem \ref{largep}} We only need to prove the necessity part, as the sufficiency part follows from  Theorem \ref{carthm}. Suppose $\mu=\{\mu_\la\}_\la$ is $L^p$-Carleson for eigenfunctions. We fix any $\xi\in S^m$. For the zonal functions $Z_{\la,\xi}$ in \eqref{zonal},  we have for $p>\frac{2m}{m-1}$
 \[\la^{\frac{m-1}2p-m}\approx \int_M|Z_{\la,\xi}|^pd\mu_\la\ge \int_{B(\xi,1/\la)}|Z_{\la,\xi}|^pd\mu_\la\approx \la^{\frac{m-1}2p}\mu_\la(B(\xi,1/\la)).\]
Thus, we have
\[\mu_\la(B(\xi,1/\la))\ls \la^{-m}.\]
So $\mu$ is relatively sparse. 

\section{Further discussions}\label{gausssect} It is open to characterize $L^p$-Logvinenko-Sereda sets and $L^p$-Carleson measures for eigenfunctions on $M$ for $1\le p\le \frac{2m}{m-1}$, even if $M$ is the standard sphere. In this section, we  give some necessary conditions on the standard sphere $S^m$ by choosing Gaussian beams on $S^m$ as test functions. We also propose a conjecture and provide some evidence. 

For $R>0$, let $\mathcal{T}_{R}(\gamma)$ be the $R$-neighborhood of a closed geodesic $\gamma$ in the sphere $S^m$. Then we  have $$\text{vol}(\mathcal{T}_{\la^{-1/2}}(\gamma))\approx \la^{-\frac{m-1}2}.$$  In this section, we fix $G_{\la,\gamma}$ to be the $L^2$-normalized Gaussian beam concentrated along the geodesic $\gamma$. Then we have for $1\le p\le \infty$ $$\|G_{\la,\gamma}\|_{L^p(M)}\approx \la^{\frac{m-1}2(\frac12-\frac1p)}.$$ Moreover, $|G_{\la,\gamma}(x)|\approx \la^{\frac{m-1}4}$ for most  $x\in \mathcal{T}_{\la^{-1/2}}(\gamma)$, and 
\[|G_{\la,\gamma}(x)|\le C_N\la^{\frac{m-1}4}(1+\sqrt\la d(x,\gamma))^{-N},\ \forall N.\]
Here $d(x,\gamma)$ is the distance from $x$ to the geodesic $\gamma$. See e.g. Sogge \cite{soggeprob}.
 \subsection{Tubular geometric control condition for Logvinenko-Sereda sets for eigenfunctions on the standard sphere}
 Let $M$ be the standard sphere $S^m$. Suppose $\mathcal{A}=\{A_\la\}_\la$ is $L^p$-Logvinenko-Sereda for eigenfunctions on $M$. We fix any closed geodesic $\gamma\subset S^m$. For any $c>0$,  we get for $N>m$
\begin{align*}
	\la^{\frac{m-1}2(\frac12-\frac1p)p}&\approx \int_M|G_{\la,\gamma}|^pdV  \ls \int_{A_\la}|G_{\la,\gamma}|^pdV\\
	& \le \int_{A_\la\cap \mathcal{T}_{c\la^{-1/2}}(\gamma)}|G_{\la,\gamma}|^pdV +\int_{M\setminus \mathcal{T}_{c\la^{-1/2}}(\gamma)}|G_{\la,\gamma}|^pdV\\
	&\le C_1 \la^{\frac{m-1}4p}\text{vol}(A_\la\cap \mathcal{T}_{c\la^{-1/2}}(\gamma))+C_N\la^{\frac{m-1}4p}\int_{r>c\la^{-1/2}}(1+\sqrt\la r)^{-N}r^{m-2}dr\\
	&\le C_1 \la^{\frac{m-1}4p}\text{vol}(A_\la\cap \mathcal{T}_{c\la^{-1/2}}(\gamma))+C_N'c^{-N+m}\la^{\frac{m-1}4p-\frac{m-1}2}.
\end{align*}
 So we can choose $c$ sufficiently large to obtain
\begin{equation}
	\text{vol}(A_\la\cap \mathcal{T}_{c\la^{-1/2}}(\gamma))\gs \la^{-\frac{m-1}2}.
\end{equation}
So a necessary condition for $\mathcal{A}$ to be $L^p$-Logvinenko-Sereda for eigenfunctions on the standard sphere $S^m$ is that there exist $c>0$ and $\rho>0$ such that for any closed geodesic $\gamma\subset S^m$ and $\la\ge1$,
\begin{equation}\label{gtube}
	\frac{\text{vol}(A_\la\cap \mathcal{T}_{c\la^{-1/2}}(\gamma))}{\text{vol}(\mathcal{T}_{c\la^{-1/2}}(\gamma))}\ge \rho.
\end{equation}
We  expect that this condition is also sufficient for $1\le p<\frac{2m}{m-1}$. Indeed, this condition implies that the inequality 
\[\int_{A_\la}|e_\la|^pdV\gs \int_M|e_\la|^pdV\] holds for $1\le p<\frac{2m}{m-1}$ whenever $e_\la$ is a Gaussian beam or zonal function. For the Gaussian beam $G_{\la,\gamma}$, we exploit the fact that $|G_{\la,\gamma}(x)|\gs \la^{\frac{m-1}4}$ for most of $x\in\mathcal{T}_{c\la^{-1/2}}(\gamma)$. For the zonal function $Z_{\la,\xi}$, we exploit the fact that  $\text{vol}(A_\la)\gs1$ and $|Z_{\la,\xi}(x)|\gs1$ for most of $x\in A_\la$. 

\begin{remark}
	 Nicolas Burq presented us with two other examples of spherical harmonics that may concentrate on geodesically curved curves on $S^2$, see \cite[Remark 5.4]{BGT}. By careful calculations, we find that  the condition \eqref{gtube} is still sufficient for these two examples when $p=2$.
\end{remark}\textbf{}
Further support for the sufficiency of  condition \eqref{gtube} comes from  control theory. It is well-known that if $\omega \subset M$ is an open set and satisfies the geometric control condition (GCC), that is, for some time $T>0$, every geodesic of length $T$ intersects $\omega$. Then the following observability inequality holds for solution $u(t,x)$ to the homogeneous wave equation on $M$:
$$
\|(u(0,\cdot), u_t(0,\cdot))\|_{L^2(M) \times H^{-1}(M)}^2 \le C_T(\omega) \int_0^T \int_{\omega} |u(t,x)|^2dV(x)dt.
$$
See e.g. \cite{BLR, BG97,RT74,HPT}. In particular, taking $u = e^{it\sqrt{-\Delta}} e_\lambda = e^{it\lambda} e_\lambda$, we obtain
$$
\|e_\lambda\|_{L^2(M)} \lesssim \|e_\lambda\|_{L^2(\omega)}.
$$
Drawing on this similarity, we refer to condition \eqref{gtube} as the \textbf{tubular geometric control condition} (TGCC). Based on the preceding discussion, we propose the following conjecture.
\begin{conj}\label{tubeconj}
	For $1\le p< \frac{2m}{m-1}$, $\mathcal{A}$ is $L^p$-Logvinenko-Sereda for eigenfunctions on the standard sphere $S^m$ if and only if $\mathcal{A}$ satisfies the tubular geometric control condition.
\end{conj}
If the conjecture holds for some  $1 \le p < \frac{2m}{m-1}$, then it implies that, the $L^p$-norm of any sequence of spherical harmonics cannot concentrate within the  $o(\lambda^{-\frac{1}{2}})$-neighborhood of the equator. 

\subsection{Tubular geometric control condition for Carleson measures for eigenfunctions on the standard sphere}
Let $M$ be the standard sphere $S^m$. Suppose a family of measures $\mu = \{\mu_\lambda\}_\lambda$ is $L^p$-Carleson for eigenfunctions on $M$. Similarly, testing Gaussian beams yields the following necessary condition: there exists a constant $C>0$ such that for any closed geodesic $\gamma \subset S^m$ and any $\lambda \ge 1$,
\begin{equation}\label{tube}
	\frac{\mu_\lambda(\mathcal{T}_{\lambda^{-1/2}}(\gamma))}{\text{vol}(\mathcal{T}_{\lambda^{-1/2}}(\gamma))} \le C.
\end{equation}

However, this condition is not sufficient for any $1<p<\infty$. Indeed, if we fix any $\xi\in S^m$ and let 
\[d\mu_\la=\la^{\frac{m+1}2}\1_{B(\xi,\la^{-1})}dV,\]
then it satisfies the condition \eqref{tube} since for any closed geodesic $\gamma\in S^m$ we have
\[\mu_\la(\mathcal{T}_{\la^{-1/2}}(\gamma))\ls \la^{\frac{m+1}2}\la^{-m}\approx \text{vol}(\mathcal{T}_{\la^{-1/2}}(\gamma)).\]
For the zonal functions in \eqref{zonal}, we have
\[\int_M|Z_{\la,\xi}|^pd\mu_\la\approx \la^{\frac{m-1}2p}\la^{\frac{m+1}2}\la^{-m}=\la^{\frac{m-1}2(p-1)},\]
and
\[\int_M|Z_{\la,\xi}|^pdV\approx \begin{cases}
1,\ \ \quad\quad\quad\quad\quad\quad1\le p<\frac{2m}{m-1}\\
\log\la,\ \ \quad\quad\quad\quad \quad p=\frac{2m}{m-1}\\
	\la^{\frac{m-1}2p-m},\ \ \quad\quad\quad\frac{2m}{m-1}<p<\infty.
\end{cases}\]
So $\mu$ is not $L^p$-Carleson for eigenfunctions on the standard sphere $S^m$ for any $1<p<\infty$.

		\bibliographystyle{plain}
		
	\end{document}